\documentclass[12pt]{amsart}
\usepackage{amsfonts}  
\usepackage{mathrsfs}
\usepackage{pictex, color,hyperref}

\newif\ifpdf
	\ifx\pdfoutput\undefined
	\pdffalse 
	\else
	\pdfoutput=1 
	\pdftrue
	\fi

	\ifpdf
	\usepackage[pdftex]{graphicx}
	\else
	\usepackage{graphicx}
	\fi

\theoremstyle{plain}
\newtheorem{Thm}{Theorem}
 
\newtheorem{Lem}{Lemma}
\newtheorem{Cor}{Corollary}
\theoremstyle{definition}
\newtheorem{Def}{Definition}
\newtheorem{Rmk}{Remark}

\def\eps{{\varepsilon}}

\begin{document}

\title{Diophantine approximation on Veech surfaces}

\author{Pascal Hubert}
\address{LATP, case cour A, Facult\'e des sciences Saint J\'er\^ome \\
Avenue Escadrille Normandie Niemen \\
13397 Marseille cedex 20, France}
\email{hubert@cmi.univ-mrs.fr}
\author{Thomas A. Schmidt}
\address{Department of Mathematics\\Oregon State University\\Corvallis, OR 97331, USA}
\email{toms@math.orst.edu}
\thanks{The first author is partially supported by Projet blanc 
ANR: ANR-06-BLAN-0038.   Both authors thank the Hausdorff Research Institute for Mathematics, and  the organizers of the special trimester on the geometry and dynamics of Teichm\"uller space for providing a very stimulating atmosphere.} 
\date{9 October 2010}

\setcounter{tocdepth}{2}

\baselineskip=17pt

\begin{abstract}
We  show that Y. Cheung's general $Z$-continued fractions  can be adapted to give approximation by saddle connection vectors  for any compact translation surface.      That is,  we show the finiteness of  his Minkowski constant for any compact  translation surface.    Furthermore,  we show that for a Veech surface in standard form,  each component of any saddle connection vector dominates its conjugates.    The saddle connection continued fractions then allow one to recognize certain transcendental directions by their developments. 
\end{abstract}


\maketitle

\section{Introduction and Main Results}   

We show that Yitwah Cheung's generalization of the geometric interpretation of regular continued fractions gives a successful method for approximation of flow directions on translation surfaces by saddle connection vectors.    Cheung \cite{C}, \cite{CHM} generalizes the work of  Poincar\'e and Klein  by replacing approximation by the integer lattice  in $\mathbb R^2$ with approximation by any infinite discrete set $Z$ of nonzero vectors with finite ``Minkowski constant'', equal to one-fourth times the supremum taken over the areas of centro-symmetric bounded convex bodies disjoint from $Z$.\\

We prove, as Cheung certainly understood, that the set of   saddle connection vectors of any translation surface  has a finite Minkowski constant.  

\begin{Thm}\label{t:goodMink}    
Let $S$ be a compact translation surface,  and $Z = \text{V}_{\text{sc}}(S)$ the set of saddle connection vectors of $S$.   Then 
\[\mu(Z) \le  \pi \,\text{vol}(S)\,\]
where $\text{vol}(S)$ is the Lebesgue area of $S$.   
\end{Thm}

 The following result is of independent interest; here, it  allows us to reach transcendence results using approximation by saddle connection vectors.    Recall that the group of matrix parts (the so-called ``derivatives'') of the oriented affine diffeomorphisms of a  compact finite genus translation surface, $S$, form a Fuchsian group, $\Gamma(S)$.  The trace field of the surface is the algebraic number field generated over the rationals by the set of traces of the elements of $\Gamma(S)$, when this group is non-trivial.  When $\Gamma(S)$ is a lattice in $\text{SL}_2(\mathbb R)$, the surface is said to be a Veech surface.   
 
\begin{Thm}\label{l:holVectorEntries}    Suppose that $S$ is a Veech surface  normalized so that:   $\Gamma(S) \subset \text{SL}_2(\mathbb K)$;  the horizontal direction is periodic; and,  both components of every   saddle connection vector of $S$ lie in  $\mathbb K$, where $\mathbb K$ is the trace field of $S$.  Then there exists a positive constant $c = c(S)$ such that  for all holonomy vectors $v = \begin{pmatrix} v_1\\v_2\end{pmatrix}$, and 
 $1\le i \le 2$ one has 
\[ |\,  v_{i}\,| \ge c \,  |\, \sigma(\, v_{i})\,|\,,\]
where  $\sigma$ varies through the set of  field embeddings of $\mathbb K$ into $\mathbb R$.   
\end{Thm}

Note that in the above,  all field embedding $\sigma$ in fact take values only in $\mathbb R$.

With $S$ as above and $Z = \text{V}_{\text{sc}}(S)$ the set of saddle connection vectors of $S$,   the $Z$-expansion of an inverse slope $\theta$ for a flow direction is defined in Section \ref{ss:theZrev}.     Theorem ~ \ref{t:goodMink}  then implies that this gives a sequence of elements $(p_n, q_n) \in \mathbb K^2$ such that    $| \theta - p_n/q_n|$ goes to zero as $n$ tends to infinity; see Lemma \ref{l:convergencePlus}.   

One criterion for a ``good'' continued fraction algorithm is that extremely rapid convergence to a real number implies that this number is transcendental.   We show that the  $Z = \text{V}_{\text{sc}}(S)$-fractions on Veech surfaces enjoy this property.

\begin{Thm}\label{t:transc}    With $S$ and $\mathbb K$ as above,  
let $D = [\mathbb K: \mathbb Q]$ be the field extension degree of $\mathbb K$ over the field of rational numbers. 
If a real number $\xi \in [0,1] \setminus \mathbb K$ has an infinite   $\text{V}_{\text{sc}}(S)$-expansion, whose
convergents $p_n/q_n$ satisfy
\[
\limsup_{n \to \infty}\,  \dfrac{\log \log q_n}{n} >  \log(\,2 D - 1\,) \,,
\]
then $\xi$ is transcendental.  
\end{Thm}
  
\subsection{Related work} 
There exist algorithms that approximate flow directions on particular translation surfaces by so-called parabolic directions, see \cite{AH}, \cite{SU}, \cite{SU2}.    
Roughly speaking, these algorithms can be viewed as continued fraction algorithms expressing real values in terms of the orbit of infinity under the action of a related Fuchsian group.  Up to finite index and appropriate normalization,  each underlying group in these examples is one of the  infinite family of Hecke triangle Fuchsian groups, \cite{Vch}.     Some 60 years ago, for each Hecke group,  D. ~Rosen \cite{R} gave a continued fraction algorithm.    Motivated in part by the use in \cite{AS} of the Rosen fractions to identify pseudo-Anosov directions with vanishing so-called SAF-invariant,  with Y. Bugeaud, in \cite{BHS} we recently gave the first transcendence results using Rosen continued fractions.   Theorem ~\ref{t:transc} is the analog of a main result there.

Each Hecke group is contained in a particular $\text{PSL}(2,  K)$ with $K$  a totally real number field.   Key to the approach  of \cite{BHS} was the fact that any element in a Hecke group  of sufficiently large trace is such that this trace  is appropriately larger than each of its conjugates over $\mathbb Q$.   This leads to a bound of the height of a convergent $p_n/q_n$ in terms of $q_n$ itself.   The LeVeque form of Roth's Theorem, in combination with a bound on approximation in terms of $q_n q_{n+1}$, 
 can then be used to show that transcendence  is revealed by exceptionally high rates of growth of the $q_n$.      We show here that all of this is possible for any Veech surface, replacing Rosen fractions by  $Z$-expansions with $Z = \text{V}_{\text{sc}}(S)$.   Key to this is our results that  (1) any nontrivial Veech group $\Gamma(S)$ has the property of the dominance of traces over their conjugates,  and (2) in the case of a Veech surface $S$,  dominance property for the group implies the dominance of components of saddle vectors over their conjugates.

We mention that it would be interesting to compare the approximation in terms of saddle connection vectors with the known instances of approximation with parabolic directions.

\subsection{Outline}    In the following section we sketch some of the disparate background necessary for our results;   in Section ~\ref{s:MinkFin} we prove the crucial result that the Minkowski constant is finite for any compact translation surface;   in Section ~\ref{s:bdHts} we show that if  $S$ is a Veech surface then  $\Gamma(S)$ has the property of dominating conjugates and from this that one can bound the heights in the $Z$-expansions, $Z = \text{V}_{\text{sc}}(S)\,$;  finally, in Section ~\ref{s:trans} we very briefly show that the arguments of \cite{BHS} are valid here:     $Z = \text{V}_{\text{sc}}(S)\,$-expansions with extremely rapidly growing denominators belong to transcendental numbers. 

\subsection{Thanks}  It is a pleasure to thank Curt McMullen for asking if the results of \cite{BHS} could hold in the general Veech surface setting.   We also thank Emmanuel Russ for pointing out the reference \cite{Ba}.
 
\section{Background}
\subsection{Cheung's $Z$-expansions}\label{ss:theZrev}  We briefly review Yitwah Cheung's definition of his $Z$-expansions --- we follow  Section 3 of \cite{CHM}, although we focus on approximation of a ray instead of a line.   Fix a discrete set $Z \subset \mathbb R^2$,  and assume that $Z$ does not contain the zero vector.    Given a positive real $\theta$,  consider the ray  emitted from the origin with slope $1/\theta$.   Our goal is to define a sequence of elements of $Z$ that approximates this ray.    

\begin{Rmk}  Note that the number that is approximated here is the {\em inverse of the slope} of the ray.   This choice accords well with the {\em projective} action of $\text{SL}_2(\mathbb R)$ on $\mathbb P^1(\mathbb R) = \mathbb R \cup \{\infty\}$.   
\end{Rmk} 

Let $u$ be the unit vector in the direction of the ray.   Denote the positive half plane of the ray by  $H_{+}(\theta) = \{ v \in \mathbb R^2\,\vert\, u \cdot v > 0\}$, and let  $Z_{+}(\theta) := Z \cap H_{+}(\theta)$.    Let $v = (p,q) \in \mathbb R^2\,$;   the difference vector  between $v$ and the vector whose endpoint is given by the intersection of $y=  x/\theta$ and $y = q$ has length of absolute value $\text{hor}_{\theta}(v) = \vert q \theta - p\vert$.    The value $q$ is the {\em height} of $v$ and $\text{hor}_{\theta}(v)$ is its {\em horizontal component}, see Figure ~\ref{zParalleloFig}.

\begin{Def}  The {\em $Z$-convergents} of $\theta$ is the set of elements of $Z$ in the half-plane of the ray  such that each minimizes the  horizontal component $\text{hor}_{\theta}(v)$ amongst elements of equal or lesser height: 
\[ \text{Conv}_{Z}(\theta) = \{ v \in Z_{+}(\theta)\,\vert\; \forall w \in Z_{+}(\theta), |w_2| \le |v_2| \implies \text{hor}_{\theta}(v) \le \text{hor}_{\theta}(w)\,\}\,.\]

The {\em $Z$-expansion} of  $\theta$ is the sequence obtained by ordering the set $\text{Conv}_{Z}(\theta)$  by height,  where we choose as necessary between elements of the same height.  
\end{Def}

Recall from \cite{CHM}  that if $Z$ contains some element of the $x$-axis,  and there are infinitely many $Z$-convergents to $\theta$, then certainly the heights of the sequence tend to infinity.  

\begin{Def}  The {\em Minkowski constant} of $Z$ is 
\[ \mu(Z) = \dfrac{1}{4}\, \sup \text{area}(\mathcal C)\,\]
where $\mathcal C$ varies through bounded, convex,  $(0,0)$-symmetric sets that are disjoint from $Z$.
\end{Def}

Finiteness of the Minkowski constant assures good approximation, see  \cite{CHM} for the proof of the following. 
\begin{Lem} (Cheung {\em et al.})  Suppose both  that $\mu(Z)$ is finite and that $Z$   contains a non-zero vector on
the $x$-axis. Then the $Z$-expansion of a direction with inverse slope $\theta$ is
infinite if and only if no element of $Z$ lies in this direction.
 \end{Lem}

\begin{figure}
\scalebox{.6}{
{\includegraphics{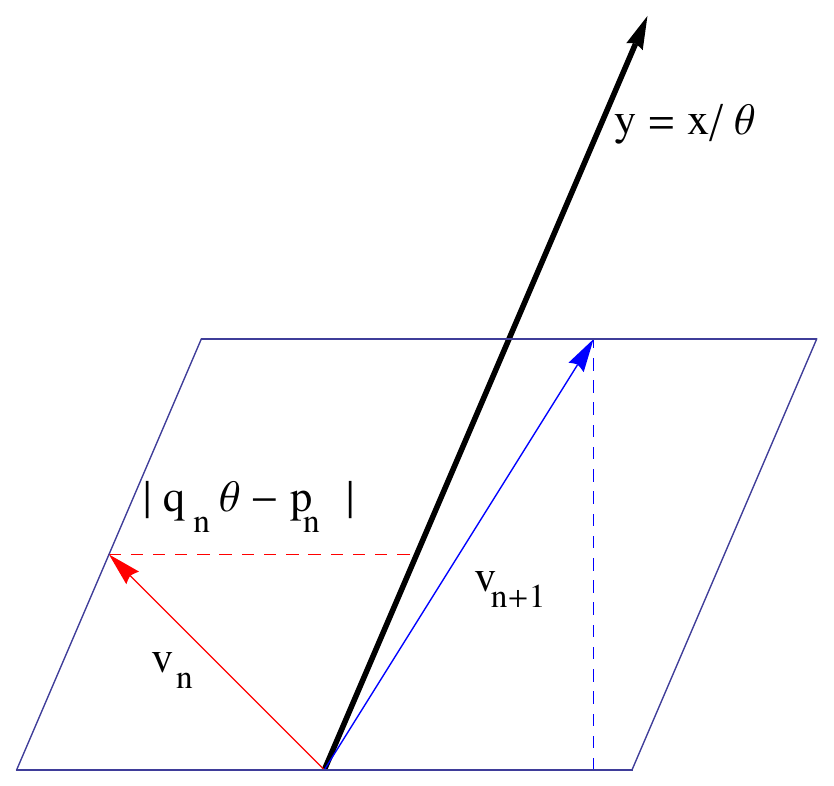}}}
\caption{Vectors $v_n$ approximate ray $y = x/\theta\,$;  parallelogram related to Lemma ~\ref{l:convergencePlus}.}
\label{zParalleloFig}
\end{figure}
 Denote the $n$-th element of the $Z$-expansion of $\theta$ by $(p_n, q_n)$.     Then the following is also shown in \cite{CHM}, see our Figure ~\ref{zParalleloFig}.  

\begin{Lem}\label{l:convergencePlus}(Cheung {\em et al.}) The $Z$-expansion  of $\theta$ satisfies 

\[ \dfrac{|p_n q_{n+1}-p_{n+1}q_n|}{ 2 \, q_n \,q_{n+1} } <  |  \theta  - p_n/ q_n | \le \mu(Z)/(q_n \, q_{n+1}) \,.\]

\end{Lem}
 
\newpage  
\subsection{Translation surfaces}

\subsubsection{Translation surfaces;     Veech surfaces}  For all of this material,  see the expository articles \cite{MT}, \cite{Z}.  
 A {\em translation surface}  is a real surface with an atlas such that, off of a finite number of points, transition functions are translations.  Here we consider only  {\em compact} surfaces, and will continued to do so without further notice.   From the Euclidean plane, this punctured surface inherits a  flat  metric,  and this metric extends to the complete surface, with (possibly removable)  conical singularities at the punctures.     Due to the transition function being translations,  directions of linear flow on a translation surface are well-defined, and Lebesgue measure is inherited from the plane.   We define $\text{vol}(S)$ to be the total Lebesgue measure of the surface.

  Post-com\-posing the coordinate function of a chart from the atlas of a translation surface with any element of $\text{SL}_2(\mathbb R)$ results in a new translation surface.  This action preserves the underlying topology,  the types of the conical singularities, and the area of the surface.     

Related to this, an {\em affine diffeomorphism} of the translation surface is a homomorphism that restricts to be a diffeomorphism on the punctured flat surface whose derivative is a {\em constant}  $2 \times 2$ real matrix.  W.~ Veech \cite{Vch}  showed that for any compact translation surface $S$,  the matrices that arise as such derivatives of (orientation- and area-preserving) affine diffeomorphisms form a Fuchsian group $\Gamma(S)$, now referred to as the {\em Veech group} of the surface.       

A {\em Veech surface}  is a translation surface such that  the group $\Gamma(S)$ is a co-finite subgroup of 
 subgroup of $\text{SL}(2, \mathbb R)$; that is,  such that the quotient of the Poincar\'e upper half-plane by $\Gamma(S)$ (using the standard 
fractional linear action) has finite hyperbolic area.   Equivalently,   $\Gamma(S)$ is a 
lattice in $\text{SL}(2, \mathbb R)$; indeed, some refer to a Veech surface as having the ``lattice property''.    

\subsubsection{Saddle connections; ergodicity of action,  parabolic directions}\label{s:sadErgPar}   A {\em saddle connection} on a translation surface $S$ is a geodesic segment connecting singularities (with no singularities on its interior).   By using the local coordinates of the translation surface, each saddle connection defines a vector in $\mathbb R^2$.   The collection of these (affine) {\em saddle connection vectors} is $\text{V}_{\text{sc}}(S)$.    That  $\text{V}_{\text{sc}}(S)$ contains an element of length at most $\sqrt{2 \, \text{vol}(S)} $ was shown by   Vorobets \cite{V}(see the proof of Proposition~  3.2 there).  

Local coordinates on the {\em stratum} that is the space of translation surfaces of fixed genus and singularities type is provided by the integral relative to the set of singularities.   In wording from \cite{EMZ},  the saddle connections cut $S$ into a collection of polygons which provide local coordinates.   The stratum then inherits a Lebesgue measure, as \cite{Z} says,  a key theorem is that if Masur and Veech:   the $\text{SL}(2, \mathbb R)$ action on translation surface preserves this measure and is ergodic on connected components of the strata.

Results of Veech \cite{Vch} imply that if  $\Gamma(S)$ is a lattice (and $S$ has singularities), then the  direction of any saddle connection vector  is a parabolic direction --- there is a parabolic element of $\Gamma(S)$ fixing a vector in this direction, and that there is some saddle connection vector in each parabolic direction.   Since a lattice in $\text{SL}_2(\mathbb R)$ has only finitely many parabolic conjugacy classes, a Veech surface has only finitely many  $\Gamma(S)$-orbits of parabolic directions.

\subsubsection{Trace field,   standard form}
  Gutkin and Judge \cite{GJ} defined the {\em trace field} of a translation surface to be the field extension of the rationals generated by the traces of derivatives of the affine diffeomorphisms of the surface; this is clearly independent of choice of a  translation surface within an $\text{SL}_2(\mathbb R)$-orbit.    A result of Gutkin and Judge (see Lemma 7.5 of \cite{GJ}) implies that the ratio of the lengths of any two saddle connection  vectors in a common parabolic direction gives an element of the trace field.   
  
  M\"oller \cite{Moe}, see Proposition ~2.6 there,  showed  that the trace field of any Veech surface is totally real (that is,  every field embedding into the complex numbers sends the field to a subfield of the real numbers).    The result holds true under weaker hypotheses,  see \cite{HL}, \cite{CS}.

      Calta and Smillie in \cite{CS} introduced a notion of {\em standard form} of a translation surface; for a Veech surface, standard form means having the vertical, the horizontal and the diagonal as parabolic directions.   They show that when a Veech surface is in standard form, the parabolic directions all lie in the trace field.     Combining this with the aforementioned result of Gutkin and Judge, one finds that,  by scaling and choice within $\text{SL}_2(\mathbb R)$-orbit, the saddle connection vectors of a Veech surface can be assumed to have components in the trace field.      Furthermore,   Kenyon and Smillie \cite{KS},  see the proof of Corollary~29 there,   argue that the $\mathbb Z$-module generated by the saddle connection vectors has a submodule of finite index that is contained in $\mathcal O_K \oplus \mathcal O_K$,  where $\mathcal O_K$ denotes the ring of integers of the trace field $K$.    In particular,  there is some $m \in \mathbb N$ such that for any $v \in \text{V}_{sc}(S)$,  the components of $m v$ are in $\mathcal O_K$.

\subsubsection{Traces of hyperbolics dominate conjugates}\label{ss:domination}   
A Fuchsian
group $\Gamma$  is said to have a {\em modular embedding} if there exists an arithmetic group $\Delta$ acting on $\mathbb H^n$  for an appropriate $n$, an inclusion $f : \Gamma \to \Delta$ and a holomorphic
embedding
$F = (F_1, . . . , F_n) : \mathbb H\to \mathbb H^n$
such that $F_1 = \text{id}$ and $F(\gamma \cdot z) = f(\gamma)\cdot F(z)$, see   \cite{CW}.

Schmutz Schaller and Wolfart \cite{SW} (see in particular their Corollary 5) show that if a Fuchsian group has a modular embedding, then its hyperbolic elements dominate their conjugates in absolute value.       For ease of discussion, let us call this the {\em domination of conjugates} property.  

M.~ M\"oller  \cite{Moe} shows that if $S$ is a Veech surface, then $\Gamma(S)$  is commensurable to a Fuchsian group with a modular embedding  (see  his Corollary 2.11).    Commensurability here means  up to finite index and $\text{SL}_2(\mathbb R)$-conjugation; it directly follows that a Veech group always has a finite index  subgroup with the  domination of conjugates property.       Using trace relations based upon the fundamental  $\text{tr}(\gamma^2) = (\, \text{tr}(\gamma)\,)^2 - 2$, it is quite likely that this can be extended to the full group.  

In the proof of Lemma ~\ref{t:domConj}, we give a more elementary  derivation, showing that any Veech group containing a hyperbolic element has the  domination of conjugates property.\\

\subsection{Approximation by algebraic numbers}

In the following, we repeat some lines of background from \cite{BHS}.

The following result was announced by Roth \cite{Roth}
and proven by LeVeque, see Chapter 4 of \cite{L}.  
(The version below is Theorem 2.5  of \cite{B}.)     
Recall that given an algebraic number $\alpha$, its {\em naive height}, denoted 
by $H(\alpha)\,$, is the largest  absolute value  of the coefficients of 
its minimal polynomial over $\mathbb Z\,$.

\begin{Thm}\label{t:Lev}(LeVeque) Let $K$ be a number field,  and $\xi$ a real algebraic number not in $K\,$.  Then, for any $\epsilon >0\,$, there exists a positive constant $c(\xi, K, \epsilon)$ such that 
\[ |\, \xi - \alpha \,| > \dfrac{c(\xi, K, \epsilon)}{H(\alpha)^{2+\epsilon}} \]
holds for every $\alpha$ in $K$.
\end{Thm}

 \bigskip

          The {\em logarithmic Weil height} of $\alpha$ lying in a number field $K$  of degree $D$ over $\mathbb Q$ is  $h(\alpha) =  \frac{1}{D}\sum_{\nu} \, 
\log^{+} || \alpha ||_{\nu}$, where $\log^{+} t$ equals $0$ if $t\le 1$ and $M_K$ denotes the places (finite and infinite ``primes'') of the field,  and $|| \cdot ||_{\nu}$ is the $\nu$-absolute value.     
This definition is independent of the field $K$ containing $\alpha$.    

The {\em product formula} for the number field $K$ is: $\prod_{\nu \in M_K} \,|| \alpha ||_{\nu} = 1$.  From this, $\forall \, \alpha, \beta \in \mathcal O_K$ with $\beta \neq 0$, such that the  ideals 
$\langle \alpha \rangle, \langle \beta \rangle$ have no common prime divisors, one finds  that  $h(\alpha/\beta) =  \frac{1}{D}\sum_{\sigma} \, 
\log^{+} \max\{| \sigma( \alpha)|, | \sigma( \beta)|\,\}$, where $\sigma$ runs through the infinite places of $K$, which we consider as field  embeddings.   Even upon dropping the relative primality condition, one finds 

\begin{equation}\label{eq:upBdHt}
h(\alpha/\beta) \le  \frac{1}{D}\sum_{\sigma} \, 
\log^{+} \max\{| \sigma( \alpha)|, | \sigma( \beta)|\,\}\,.
\end{equation}

The two heights are related by 
\begin{equation}\label{eq:heightsReltd}    
\log H(\alpha) \le \deg(\alpha) \, (\,  h(\alpha) + \log 2\,)\,, 
\end{equation}
for any non-zero algebraic number $\alpha$,
see Lemma ~3.11 in \cite{W}.  
\\

Finally, recall that  standard transcendence notation includes the use of $\ll$ and $\gg$ to denote inequality with implied constant.

\section{Minkowski constants}\label{s:MinkFin}  

\subsection{Minkowski constants in strata}

The Minkowski constant of the nonzero holonomy of a translation surface defines a function that may be of true interest.   The following shows that it has properties in common with the Siegel-Veech invariants (see \cite{EMZ}).

\begin{Lem}\label{l:minkAsInvariant}      The function assigning to a translation surface $S$ the Minkowski constant of the set of saddle connection vectors, 
$S \mapsto \mu(\, \text{V}_{\text{sc}}(S)\,)$,   is constant on $\text{SL}_2(\mathbb R)$-orbits.  
\end{Lem}
\begin{proof}    The action of $\text{SL}_2(\mathbb R)$ on translation surfaces sends  saddle connection vectors to saddle connection vectors.    But, this standard action sends the collection of  bounded,  convex sets that are symmetric about the origin to itself.  
\end{proof}

The following is now implied by the ergodicity of the $\text{SL}(2, \mathbb R)$ action. 
 
\begin{Cor}\label{l:minkDense}     Any connected component of the moduli space of abelian differentials of a given signature has a subset of full measure on which the Minkowski constant   is constant.  
\end{Cor}

 We give an example where the Minkowski constant is small.   See Figure ~ 46 of \cite{Z} for a representation of the surface in question. 
 
\begin{Lem}\label{l:minkSmall}     Consider the translation surface $S$ given by the $L$-shaped square-tiled surface of three tiles.  Then    $\mu(\,\text{V}_{\text{sc}}(S)\,) =  \text{vol}(S)/3$. 
\end{Lem}

\begin{proof}   We may assume that $S$ has area 3.   One easily finds that $\Gamma(S)$ is the Theta group,  the subgroup of the modular group generated by $z \mapsto z+2$ and $z \mapsto -1/z$.   The entries of an element $\begin{pmatrix} a&b\\c&d\end{pmatrix}$ of this group  satisfy   $a \equiv d \equiv -b \equiv -c \pmod 2$.   One immediately finds that $S$ is in standard form:  there are visibly connection vectors in the horizontal, vertical and diagonal directions.   Indeed, these are primitive vectors in  the full lattice $\mathbb Z^2$,  and Theta acts so as to give that $\text{V}_{\text{sc}}(S)$ consists of all of the primitive vectors.   Hence,  $\mu(\, \text{V}_{\text{sc}}(S)\,) = \mu(\,\mathbb Z^2\,) = 1$.   
\end{proof}

\subsection{Finiteness of Minkowski constants }  Key to convergence of Cheung's $Z$-approximants is his hypothesis that the Minkowski constant $\mu(Z)$ is finite.        Note that Theorem ~\ref{t:goodMink}  justifies the statement in Corollary 3.9 of \cite{CHM}.

\begin{proof}(of Theorem ~\ref{t:goodMink})  Fix any bounded convex region $\mathcal C$ that is symmetric about the origin in the plane.      By a theorem of  Fritz John, see say \cite{Ba} for a discussion,    the  ellipse (symmetric about the origin) $\mathcal E$  of maximal area interior to  $\mathcal C$ is such that the scaled ellipse $\sqrt{2}\,  \mathcal E$  contains $\mathcal C$.   It follows that $\text{area}(\mathcal E) \ge \text{area}(\mathcal C)/2$.   

  Now, there is $A \in \text{SL}(2, \mathbb R)$ taking $\mathcal E$ to a circle. If  $ \text{area}(\mathcal C) \ge 4 \pi \,\text{vol}(S)$, then $A\cdot \mathcal E$ contains any vector of length less than or equal to $\sqrt{2 \, \text{vol}(S)}$.    But,   $A\cdot \text{V}_{\text{sc}}(S) = \text{V}_{\text{sc}}(A\cdot S)$ has a saddle connection vector of length at most $\sqrt{2 \, \text{vol}(A \cdot S)} = \sqrt{2 \, \text{vol}(S)}$, where we have used the bound of Vorobets for the length of the shortest saddle connection mentioned in Section ~\ref{s:sadErgPar} .   Therefore,   $A \cdot \mathcal C$ contains a saddle connection vector of $A \cdot S$ and hence  $\mathcal C$ contains an element of $\text{V}_{\text{sc}}(S)$.     We conclude that $ \mu(\,\text{V}_{\text{sc}}(S)\,) \le  \pi \,\text{vol}(S)$. 
\end{proof} 

 \section{Bounding the height of convergents}\label{s:bdHts}
In the background discussion of Section ~\ref{ss:domination},  we sketched a sequence of results from the literature that should  imply that when $S$ is a Veech surface the traces of hyperbolic elements in $\Gamma(S)$ dominate their conjugates.      Here we  give a more elementary proof, with weaker hypotheses.    
 
\begin{Lem}\label{t:domConj}   Let $M$ be a  hyperbolic element in the Veech group of a translation surface.   Then the trace of $M$ is as least as large in absolute value as any of its images under the field embeddings of the trace field of the translation surface. 
 \end{Lem}
 
\begin{proof}  Recall that $M$ corresponds to a pseudo-Anosov map, say $\phi$.    The action of $\phi$ on the integral homology of the surface gives an integral matrix,  Perron-Frobenius shows that there is a dominant eigenvalue, $\lambda$, the dilatation of $\phi$.    The minimal polynomial of $\lambda$ divides the characteristic polynomial of the action of $\phi$ on the homology.     Thus,  the Galois conjugates of $\lambda$ are also eigenvalues of the action.   Hence,  $\lambda$ dominates its conjugates over $\mathbb Q$.   

  Now,  $M$ being in the group of the given flat surface allows one to deduce that the eigenvalues of $M$ are $\lambda$ and $1/\lambda$. 
Thus, the trace of $M$ is $t = \lambda + \lambda^{-1}$.     Now,  any field embedding of $\mathbb Q(t)$ into $\mathbb C$ lifts to field embeddings of    $\mathbb Q(\lambda)$ into $\mathbb C$.   Thus,   we may write $\sigma(t) = \sigma(\lambda) + \sigma(\lambda)^{-1}$ for the image of this trace under any  field embedding of $\mathbb Q(t)$.    The function $x \mapsto x + 1/x$ is increasing on $[1, +\infty)$; thus,  from $\lambda > | \sigma(\lambda)|$, we  conclude
\[ \vert \sigma(t) \vert = \vert \sigma(\lambda) + \sigma(\lambda)^{-1} \vert \leq \vert \sigma(\lambda)\vert + \vert \sigma(\lambda)\vert ^{-1} < \lambda + \lambda^{-1}\,.
\]
 
 Finally, recall that \cite{KS} (Theorem 28) shows $\mathbb Q(t)$ already gives the full trace field of the surface.  
\end{proof} 
 
\bigskip

 We now show that if a Fuchsian subgroup $\Gamma$ of the determinant one matrices over a number field is known to have traces that dominate their conjugates, and the subgroup includes a translation, then every entry of every element of $\Gamma$ dominates its conjugates.   We do this by refining   
arguments of \cite{BHS} (see Lemma ~3.1  there).

\begin{Lem}\label{l:bigElmts}     Fix a totally real number field $\mathbb K$.   Suppose that $\Gamma \subset \text{PSL}_2(\mathbb K)$ contains  a  parabolic element  \[  \begin{pmatrix} 1&\lambda\\ 0&1\end{pmatrix} \,,\]
and set $c_1 = \min\{(|\,\sigma(\lambda)/\lambda)\,|\,)\}\,$, where  $\sigma$ varies through the set of  field embeddings of $\mathbb K$ into $\mathbb R$.    Suppose further that for all $M \in \Gamma$ whose trace is sufficiently large in absolute value, that for  all $\sigma$ one has  
\[ |\, \text{tr}(M)\,| \ge  |\, \sigma(\,\text{tr}(M)\,)\,|\,.\] 
 Then  for all 
\[ A = \begin{pmatrix} a_{11}&a_{12}\\a_{21}&a_{22}\end{pmatrix} \in \Gamma\,,\]
and for all $\sigma$ and for all $1 \le i,j \le 2$, one has 
\[ |\, a_{ij}\,| \ge c_1 \,  |\, \sigma(\,a_{ij}\,)\,|\; \mbox{if}\;\;  i \neq j\]
\[ \mbox{and }\;  |\,a_{ii}\,| \ge c_{1}^{2} \;  |\, \sigma(\,a_{ii}\,)\,|\;\; \,.\]
\end{Lem}
\begin{proof}   For ease of notation, let $P =  \begin{pmatrix} 1&\lambda\\ 0&1\end{pmatrix}$.   Let $A \in \Gamma$ be arbitrary.   The trace of $P^n A$ is $a_{11} + a_{22} + n \lambda a_{21}$; thus, either $a_{21} = 0$,  or upon letting $n$ tend to infinity this trace is eventually   large in absolute value.      Thus,   we find that $a_{21}$ is at least $c_1$ times the absolute value of any of its conjugates.   Considering $A P^n$ similarly gives that $a_{12}$ is at least $c_1$ times the absolute value of any of its conjugates.      

Now,  the $(1,2)$-element of $A P^n$ is $n a_{11} \lambda + a_{12}$ and as we considering arbitrary elements in $\Gamma$ above, this must be at least $c_1$ times the absolute value of any of its conjugates.   We thus find that $|\,a_{11}\,| \ge c_{1}^{2} \;  |\, \sigma(\,a_{11}\,)\,|$.   Replacing $A$ by $A^{-1}$ shows that also $a_{22}$ has this property. 
\end{proof}
 
\bigskip 
We now prove the announced main result stating that each component of any saddle connection vector of a Veech surface is appropriately larger than any of its conjugates.

\begin{proof}(of  Theorem ~\ref{l:holVectorEntries})\qquad   A translation surface has only  finitely many singularities,  and hence only finitely many saddle connection vectors in any given direction.      
Since $S$ is a Veech surface we have both that  each non-zero holonomy vector lies in some parabolic direction and that there are only finitely many $\Gamma$-orbits of parabolic directions.     We choose a representative direction from each of these orbits,   and let $\mathcal V$ be the set of all saddle connection vectors in these chosen directions.       

 Since $S$ is in standard form,  the (positively oriented) horizontal is certainly  a parabolic direction for $S$.   In particular,   $\Gamma$ has an element of the form $\begin{pmatrix} 1&\lambda\\ 0&1\end{pmatrix}$; but since $S$ is a Veech surface, also the traces of hyperbolic elements in $\Gamma$ dominate their conjugates, and thus  Lemma ~\ref{l:bigElmts} holds.       We also can and do assume that in the above construction of $\mathcal V$, that the horizontal direction is chosen to represent its   $\Gamma$-orbit.  

 Let  $c' = \min \{ |\, \sigma(\, v_i )\,|/ |\,  v_i \,|\,\}\,$, with the minimum taken over all horizontal $v \in \mathcal V$, , $i \in \{1,2\}$, and embeddings $\sigma$.  For  $v \in \mathcal V$, there is $P \in \Gamma$ such that $P v = v$. Since $P$ is upper diagonalizable,   we can find some vector $w$ such that $Pw = v + w$.   We can then express $e_1 = \alpha v + \beta w$ for some real number $\alpha, \beta$, with $\beta \neq 0$ when $v$ is non-horizontal.   Note that since $Pe_1 = e_1 + \beta v$ and $P \in \Gamma$, we must have that $\beta \in \mathbb K$.  
Choosing such a $w$ for each $v\in \mathcal V$,  let $c'' = \min \{ |\, \sigma(\, \beta )\,|/ |\,  \beta \,|\,\}\,$, over all non-horizontal $v \in \mathcal V$, and embeddings $\sigma$.       Finally,  set $c=  c''' \, c_{1}^{2}$ with $c''' = \min\{c', c''\}$.
  
Now,  if $h$ is an arbitrary saddle connection vector of $S$, then there exists some $A \in \Gamma$ and $v \in \mathcal V$ such that $h = Av$.  If $v = \alpha e_1$ is horizontal,  then $h$ is the multiple by $\alpha$ of the first column of $A$.    Our result clearly holds in this case.  
 Otherwise, with notation as above,  induction gives $P^n e_1 = e_1 + n \beta v$, and thus $AP^n e_1 = Ae_1 + n \beta h$.   The left hand side is the first column of an element of $\Gamma$,   thus our standard argument  allows us to conclude that each of $\beta h_i$ with $i=1,2$ is greater in absolute value than $c_{1}^{2}$ times any of its conjugates.      Here also we find that each of $h_i$ is greater in absolute value than $c$ times any of its conjugates.   

Finally, by the finiteness of $\mathcal V$ one easily verifies that $c$ may be taken to depend only on $S$. 
\end{proof}

 We now can bound the heights of the saddle connection approximants.
 
\begin{Lem}\label{l:heightBd}   Fix a Veech surface $S$ in standard form, with trace field $\mathbb K$.  Let $D$ denote the field extension degree  
$[\mathbb K:\mathbb Q\,]\,$. 
There exists a constant $c_2 = c_2(S)$ such that for all $x \in [0,1]$ of infinite $\text{V}_{\text{sc}}(S)$-expansion  $(p_n/q_n)_{n\ge 1}\,$, 
\[
H(p_n/q_n) \le c_2 q_n^D.
\]
\end{Lem}
 
\begin{proof}   There is a positive integer $m$ depending only on $S$ such that $m v \in \mathcal O_{K}^{2}$ for all $v \in \text{V}_{\text{sc}}(S)$.   Writing $(p_n, q_n) = (\alpha/m, \beta/m)$ with $\alpha, \beta \in \mathcal O_K$,   Equation ~\eqref{eq:upBdHt} gives $h(p_n/q_n) \le  \log m +  \frac{1}{D}\sum_{\sigma} \, \log^{+} \max\{| \sigma( p_n)|, | \sigma( q_n)|\,\}$.   With $c$ as in Theorem ~\ref{l:holVectorEntries}  we find $h(p_n/q_n)  \le  \log c \,m +  \log \max\{ |p_n |,  q_n\,\} \le  \log c' c\, m +   \log  q_n$  for $c'$ depending only on $S$.  
Now,   Equation ~\eqref{eq:heightsReltd} implies the result. 
\end{proof}

\section{Transcendence with $Z$-fractions}\label{s:trans}

We prove our transcendance result in the traditional manner:  by showing that the sequence of denominators of convergents to an algebraic number cannot grow too quickly.     The ingredients are the result of Roth--LeVeque     and the convergence bound with a denominator of $q_{n}q_{n+1}$.

\begin{proof} (of Theorem ~\ref{t:transc} ) \quad  Let $\eps$ be a positive real number.   Let $\zeta$ be an algebraic number
having an infinite $Z= \text{V}_{\text{sc}}(S)$-expansion with convergents $r_n / s_n$.

By the Roth--LeVeque Theorem \ref{t:Lev},  we have 
\[
|\zeta - r_n / s_n| \gg  H(r_n / s_n)^{- 2 - \eps}, \quad \hbox{for $n \ge 1$}.
\]
And, hence by Lemma ~\ref{l:heightBd},  
for $n \ge 1$, we have $|\zeta - r_n / s_n| \gg s_n^{- 2D - D \eps}\,$.

The key to the proof is provided 
by applying Lemma ~\ref{l:convergencePlus} and thus finding that 
 there exists a constant $c_3$ (independent of $n$) such that 
\[
s_{n+1} < c_3 s_n^{2D - 1 + D \eps}\,.
\]
 
Thereafter,  standard manipulations,  as in the proof of  Theorem~1.1 of \cite{BHS} give the result. 
\end{proof}


\end{document}